\documentclass[a4paper,11pt,pdf]{amsart}
\usepackage{enumerate, amsmath, amsfonts, amssymb, amsthm, thmtools, wasysym, graphics, graphicx, xcolor, frcursive,comment,bbm}
\usepackage[colorlinks=true,citecolor=cyan,backref=page]{hyperref}

\usepackage{etex}
\usepackage{lscape}
\usepackage{tikz-cd}

\makeatletter
\newcommand{\thickhline}{%
    \noalign {\ifnum 0=`}\fi \hrule height 1pt
    \futurelet \reserved@a \@xhline
}

\definecolor{darkblue}{rgb}{0.0,0,0.7} 
 
\definecolor{darkred}{rgb}{0.7,0,0} 
\usepackage{hyperref}
\usepackage[all]{xy}
\usepackage[T1]{fontenc}
\usepackage{adjustbox}

\usepackage{vmargin}            
\setmarginsrb{3cm}{2.5cm}{3cm}{2.5cm}{0cm}{0.6cm}{0cm}{0cm}

\usepackage{caption,lipsum}
\usepackage{graphicx}                  
\usepackage{pstricks,pst-plot,pst-text,pst-tree,pst-eps,pst-fill,pst-node,pst-math}
\usepackage{setspace}
\usepackage{multicol}

\usepackage{etex}

\newtheorem{theorem}{Theorem}[section]
\newtheorem{prop}[theorem]{Proposition}
\newtheorem{lemma}[theorem]{Lemma}

\theoremstyle{definition}

\newtheorem{rmq}[theorem]{Remark}
\newtheorem{exple}[theorem]{Example}

\numberwithin{equation}{section}

\title[Elements of minimal length and Bruhat order on fixed point cosets]{Elements of minimal length and Bruhat order on fixed point cosets of Coxeter groups}

\author{Nathan Chapelier-Laget}
\address{University of Sydney, School of Mathematics and Statistics}

\author{Thomas Gobet}
\address{Institut Denis Poisson, CNRS UMR 7350, Faculté des Sciences et Techniques, Université de Tours, Parc de Grandmont, 
37200 TOURS, France}

\begin{document}
\maketitle

\thispagestyle{empty}

\begin{abstract}
We study the restriction of the strong Bruhat order on an arbitrary Coxeter group $W$ to cosets $x W_L^\theta$, where $x$ is an element of $W$ and $W_L^\theta$ the subgroup of fixed points of an automorphism $\theta$ of order at most two of a standard parabolic subgroup $W_L$ of $W$. When $\theta\neq\mathrm{id}$, there is in general more than one element of minimal length in a given coset, and we explain how to relate elements of minimal length. We also show that elements of minimal length in cosets are exactly those elements which are minimal for the restriction of the Bruhat order.  
\end{abstract}

\section{Introduction}

When studying Coxeter groups, one often encounters subgroups which themselves admit a structure of Coxeter group. Although there do not seem to exist a general theory of "Coxeter subgroups" of Coxeter groups, at least several important families of subgroups are known to admit canonical structures of Coxeter groups: this includes (standard) parabolic subgroups, or more generally reflection subgroups~\cite{Deodhar, Dyer_subgroups}. Another family is given by subgroups obtained as fixed points of automorphisms of the Coxeter-Dynkin diagram. 

Let $(W,S)$ be a Coxeter system. In the most basic of the aforementioned situations, one considers a subset $J\subseteq S$, and defines $W_J$ as the subgroup of $W$ generated by the elements of $J$. The pair $(W_J, J)$ is again a Coxeter system, with Coxeter-Dynkin diagram obtained from the diagram of $W$ by removing the vertices corresponding to generators in $S\backslash J$. In this situation, every coset $x W_J$ admits two basic properties, namely \begin{enumerate} \item \label{p1} There is a unique element $x^J\in x W_J$ which has minimal length among all elements in $x W_J$,
\item \label{p2} For all $y\in x W_J$, one has $x^J \leq y$, where $\leq$ denotes the strong Bruhat order on $W$.
\end{enumerate}

In fact, for $y\in x W_J$, one has the stronger statement that $x^J$ is below $y$ for the right weak order; nevertheless, there are generalizations in which this is too much to expect. For instance, Dyer extended these properties to the far more general setting of reflection subgroups of Coxeter groups~\cite[Theorem 1.4]{Dyer_Bruhat}, that is, to the case where $W_J$ is replaced by any subgroup $W'$ of $W$ generated by a subset of the set $T=\bigcup_{w\in W} w S w^{-1}$ of \textit{reflections} of $W$. In this setting, property~\ref{p2} is only valid for the strong Bruhat order, not for the right weak order in general.  

\medskip

The purpose of this article is to study the analogue of Properties~\ref{p1} and~\ref{p2} above for another class of Coxeter subgroups of Coxeter groups, given by fixed points subgroups of automorphisms squaring to the identity of the Coxeter-Dynkin diagram of (a standard parabolic subgroup of) $W$. 

To be more precise, let $(W,S)$ be a Coxeter system and $L\subseteq S$ be a subset. Let $W_L$ be the corresponding standard parabolic subgroup of $W$. Let $\theta$ be an automorphism of $W_L$ such that $\theta(L)=L$. Then $$W_L^\theta:=\{w\in W_L \ \vert \ \theta(w)=w\}$$ admits a structure of Coxeter group; this was observed by Steinberg for finite Weyl groups~\cite[Section 11]{Steinberg}, and later generalized to arbitrary Coxeter systems independently by Hée~\cite{Hee_1} and Mühlherr~\cite{muhlh} (see also Lusztig~\cite[Appendix]{Lus_book}). The simple system is obtained as follows. First partition $L$ into orbits $(J_i)_{i\in I}$ under the action of $\theta$. Then, whenever $J_i$ is such that the standard parabolic subgroup $W_{J_i}$ is finite, consider its longest element. The simple system $S_L^\theta$ consists of all these elements. In what follows, we will restrict ourselves to the case where $\theta^2=\mathrm{id}$, in which case every $W_{J_i}$ is either of type $A_1$ or dihedral.  

Note that, when $W$ is irreducible and $L=S$, there may not be a lot of nontrivial automorphisms $\theta$ of the Coxeter-Dynkin diagram, but for $L\neq S$ the subgroup $W_L$ may not be irreducible, yielding many possible automorphisms permuting irreducible components that are isomorphic as Coxeter groups. Such a situation arose in work of Chaput, Fresse and the second author~\cite[Section 3]{CFG}, where a study of the analogues of Properties~\ref{p1} and~\ref{p2} of cosets $x W_L^\theta$ for subsets $L$ of a certain form was an important step in the understanding of a partial order defined on the quotient $W/W_L^\theta$, which in type $A$ describes a "Bruhat-like order" given by inclusion of certain nilpotent orbit closures.

Unfortunately, unlike for the case of standard parabolic subgroups (or more generally reflection subgroups), there is \textit{not} a unique element of minimal length in a given coset in general. The main results addressing the analogues of Properties~\ref{p1} and~\ref{p2} may be summarized as follows: 

\begin{theorem}[Relation between elements of minimal length in a given coset]\label{thm_min}
Let $u,v\in W$, $y\in W_L^\theta$ such that $v=uy$ and $u,v$ are both of minimal length in $u W_L^\theta= v W_L^\theta$. Let $y_1 y_2 \cdots y_k$ be an $S_L^\theta$-reduced expression of $y$ in $W_L^\theta$. Then we have $$\ell(u)=\ell(uy_1)=\ell(uy_1 y_2)=\cdots =\ell(u y_1 \cdots y_{k-1})=\ell(v).$$
In other words, for all $i=1, \dots, k-1$, we have that $u y_1 \cdots y_i$ is of minimal length in $uW_L^\theta=v W_L^\theta$.  
\end{theorem}

The main ingredient for proving Theorem~\ref{thm_min} is the following proposition, which will also be useful for the proof of Theorem~\ref{main_mm} below addressing the analogue of Property~\ref{p2}: 

\begin{prop}\label{prop_esc}
Let $u, w\in W$ such that $u$ is of minimal length in $w W_L^\theta$. Let $z\in W_L^\theta$ such that $w=uz$ and let $x_1 x_2 \cdots x_k$ be an $S_L^\theta$-reduced expression of $z$. For all $i=0, \dots, k-1$, exactly one of the following two situations occurs:
\begin{itemize}
	\item either $\ell(u x_1 \cdots x_i)=\ell(u x_1 \cdots x_{i+1})$, 
	\item or $u x_1 \cdots x_i < u x_1 \cdots x_{i+1}$.
\end{itemize}
In particular, if $x_{i+1}$ is a reflection of $W$, then we are in the second situation, while the first situation can only occur if $x_{i+1}$ is not a reflection of $W$. 
\end{prop}

The analogue of Property~\ref{p2} is then given by the following statement:

\begin{theorem}[Elements of minimal length are minimal for the strong Bruhat order]\label{main_mm}
Let $x\in W$. There is an element $w\in W$ which is of minimal length in $x W_L^\theta$, and such that $w \leq x$. In other words, the elements of minimal length in any coset $x W_L^\theta$ are precisely those elements which are minimal with respect to the restriction of the strong Bruhat order $\leq$ on $W$ to $x W_L^\theta$. 
\end{theorem}

\textbf{Acknowledgements.} We thank Pierre-Emmanuel Chaput and Lucas Fresse for a careful reading of the manuscript, and several useful comments and remarks. We thank Christophe Hohlweg for useful discussions. 

\section{Preliminaries and notation}

Let $(W,S)$ be a Coxeter system (with $S$ finite) with set of reflections $T=\bigcup_{w\in W} w S w^{-1}$, let $L\subseteq S$, and let $\theta$ be a diagram automorphism of $L$. It induces an automorphism of the standard parabolic subgroup $W_L$, which we still denote $\theta$. It is well-known that the subgroup $$W_L^\theta:=\{ w\in W_L \ | \ \theta(w)=w\}$$ of $\theta$-fixed elements of $W_L$ admits a structure of Coxeter group (see for instance~\cite{Hee_1, muhlh}). The generators as Coxeter group are given by the following set. Let $K\subseteq L$ be an orbit of the action of $\theta$. If $W_K$ is finite, let $w_0^K$ denote its longest element. The set of Coxeter generators of $W_L^\theta$ is given by the set of all such elements. We denote by $S_L^\theta$ the set of generators of $W_L^\theta$ as a Coxeter group. We denote by $\ell$ the classical length function on $W$. 

Note that the elements of $S_L^\theta$ are \textit{not} elements of $S$ in general. For instance, if $W$ has type $A_1 \times A_1$ with Coxeter generators $s$ and $t$, $L=S$ and $\theta$ exchanges $s$ and $t$, then $W_L^\theta=W^\theta$ has type $A_1$, with Coxeter generator $st=ts$. 

We will furthermore assume that $\theta$ satisfies $\theta^2=\mathrm{id}$. In this case, the elements of $S_L^\theta$ are longest elements of finite standard parabolic subgroups of $W$ of type $A_1$ or dihedral. In particular, if such elements have odd length, then they are reflections of $W$. It will be useful to distinguish the elements of $S_L^\theta$ depending on the parity of their length. We thus write $S_L^\theta=\Theta\overset{\cdot}{\cup} \Theta_T$, where \begin{align*}\Theta&:=\{ x\in S_L^\theta~|~\ell(x)~\text{is even}\}=\{ x\in S_L^\theta~|~x\notin T\},\\
\Theta_T&:=\{ x\in S_L^\theta~|~\ell(x)~\text{is odd}\}=\{ x\in S_L^\theta~|~x\in T\}.\end{align*}

For $u\in W$, we denote by $\mathrm{Min}(u)\subseteq W$ the set of elements of minimal length in $u W_L^\theta$, that is, the set $$\{ v\in u W_L^\theta~|~\ell(w)\geq \ell(v)~\text{for all}~w\in u W_L^\theta\}.$$
Let $$\mathcal{M}=\bigcup_{w\in W} \mathrm{Min}(w)$$ denote the set of elements which are of minimal length in their coset.   

\medskip

We have (see~\cite[Proposition 3.5]{muhlh})
\begin{prop}\label{prop_muhlh}
Let $x\in W_L^\theta$ and $x_1, x_2, \dots, x_k\in S_L^\theta$ such that $x_1 x_2 \cdots x_k$ is an $S_L^\theta$-reduced expression of $x$. Then $$\ell(x)=\sum_{i=1}^k \ell(x_i).$$
\end{prop}

We denote by $\leq$ the (strong) Bruhat order on $W$. Recall that it is defined as the transitive closure of the relation $x < xt$ whenever $x\in W$, $t\in T$, and $\ell(x)<\ell(xt)$. One has the following characterization, which we will use extensively in the next sections (see for instance~\cite[Corollary 2.2.3]{BB})

\begin{prop}
Let $u,v\in W$. The following are equivalent
\begin{enumerate}
	\item $u \leq v$, 
	\item There is a reduced expression of $v$ admitting a reduced expression of $u$ as a subword, 
	\item Every reduced expression of $v$ admits a reduced expression of $u$ as a subword. 
\end{enumerate}
\end{prop}

Note that by subword we mean that some letters may not be consecutive in the bigger expression. 

Also recall that, given $J\subseteq S$, the subgroup $W_J$ of $W$ generated by $J$ is a Coxeter system with simple system $J$. Denoting $$W^J=\{w\in W \ \vert \ \ell(ws)>\ell(w)~\forall s\in J\},$$ every $w\in W$ admits a unique decomposition $w=w^J w_J$ with $w^J\in W^J$ and $w_J \in W_J$, and it satisfies $\ell(w)=\ell(w^J)+\ell(w_J)$. See for instance~\cite[Section 2.4]{BB}. In particular, every coset $x W_J$ admits a unique element of minimal length $x_0$, and for every $y\in x W_J$, one has $x_0 \leq y$. 

To each element $w\in W$, consider its set $N(w)$ of (right) inversions, which is a subset of $T$ defined by $$N(w)=\{t\in T \ \vert \ \ell(wt)<\ell(w)\}=\{t\in T \ \vert \ wt<w\}.$$

Recall that $|N(w)|=\ell(w)$ and that for all $x,y\in W$, we have $$N(xy)=N(y)\Delta (y^{-1} N(x) y),$$ where $\Delta$ denotes the symmetric difference.

\bigskip

We now prove three Lemmatas that will be useful in the proofs of the main results: 

\begin{lemma}\label{lem_corr}
For all $x\in S_L^\theta$ and $u, w\in W$ such that $u \leq w$ and $\ell(w)=\ell(wx)$, there is $v\in \{u, ux\}$ such that $\ell(v)\leq \ell(u)$ and $v \leq wx$. 
\end{lemma}

\begin{proof}
	Let $u,w\in W$ and $x\in S_L^\theta$ such that $u\leq w$ and $\ell(wx)=\ell(w)$. Since $\theta^2=\mathrm{id}$, this forces $x$ to be the longest element in a dihedral standard parabolic subgroup $W_I$ where $I=\{s,t\}$ and $I$ is of type $I_2(2k)$ for some $k\geq 1$; indeed, in all the other cases, $x$ has to be the longest element in a standard parabolic subgroup of type $A_1$ or $I_2(2k+1)$, hence it is a reflection, hence $\ell(wx)\neq \ell(w)$. There are exactly two distinct elements $w_1, w_2\in W_I$ of length $k$, and they satisfy $w_1^2=x=w_2^2$ if $k$ is even (in which case $w_1=w_2^{-1}$) and $w_1 w_2=x=w_2 w_1$ if $k$ is odd (in which case $w_1$ and $w_2$ are reflections). In all cases we have $w_1 x=w_2$ and $w_2 x=w_1$. We have $u^I \leq w^I$ because the map $x \mapsto x^I$ preserves the (strong) Bruhat order (see~\cite[Proposition 2.5.1]{BB}) . 
	
	The condition $\ell(w)=\ell(wx)$ yields $\ell(w_I x)=\ell(w_I)$, which forces $w_I$ to lie in $\{w_1, w_2\}$, say $w_I=w_1$ (the roles of $w_1$ and $w_2$ are symmetric). Consider the decomposition $u=u^I u_I$. If $\ell(u_I)>k$, then $\ell(u_I x) < k$ and hence  the unique reduced expression of $u_I x$ is a subword of the unique reduced expression of $w_2$. We thus have $ux = u^I u_Ix \leq w^I w_2$, but $w^I w_2= w^I w_1x=w^I w_I x=wx$. We thus have the result with $v=ux$, since we also have $$\ell(v)=\ell(ux)= \ell(u^I)+\ell(u_I x)< \ell(u^I)+k<\ell(u^I)+\ell(u_I)=\ell(u).$$  
If $\ell(u_I)<k$, then the unique reduced expression of $u_I$ is a subword of the unique reduced expression of $w_2$. We thus have $u=u^I u_I \leq w^I w_2$, but $w^I w_2= w^I w_1 x=w^I w_I x=wx$. We thus get the result with $v=u$. It remains to treat the case where $\ell(u_I)=k$, that is, where $u_I\in \{w_1, w_2\}$. If $u_I=w_1$, then $ux= u^I u_I x= u^I w_1 x \leq w^I w_1 x =wx$, hence we get the result with $v=ux$ (also using that $\ell(u_Ix)=\ell(u_I)$), while if $u_I=w_2$, we have that $u=u^I u_I=u^I w_2 \leq w^I w_2=w^I w_1x=wx$, hence we get the result with $v=u$.
\end{proof}

\begin{lemma}\label{lem_long_gen}
	Let $u\in W$ and $x\in S_L^\theta$. If $\ell(u)< \ell(ux)$, then $u<ux$.
\end{lemma}

\begin{proof}
	We have $x=w_{0,I}$ for some $I\subseteq L$, where $w_{0,I}$ is the longest element in the finite standard parabolic subgroup $W_I$; since $\theta$ has order two, we have $|I|=1$ or $2$. If $|I|=1$, then $x\in S$ and we have $u < ux$. Hence we can assume that $|I|=2$, say $I=\{s,t\}$. Let $u=u^I u_I$ be the decomposition of $u$ in $W^I W_I$. We have $ux=u^I u_I w_{0,I}$ and $\ell(u)=\ell(u^I)+\ell(u_I)$, $\ell(ux)=\ell(u^I)+\ell(u_I w_{0,I})$. Hence, setting $u_I':=u_I w_{0,I}$, we deduce from the assumption that $\ell(u_I)<\ell(u_I')$. Since $u_I, u_I'\in W_I$ which is a dihedral Coxeter group, we get $u_I < u_I'$, hence $u=u^I u_I < u^I u_I'=ux$, which concludes the proof.
\end{proof}

\begin{lemma}\label{lem_ref_commute}
	Let $(W,S)$ be a Coxeter system and let $t, t'\in T$ with $t\neq t'$ and $tt'=t't$. Then $t\notin N(t')$. 
\end{lemma}

\begin{proof}
	Let $s_1 s_2 \cdots s_{k-1} s_k s_{k-1}\cdots s_2 s_1$ be a palindromic $S$-reduced expression of $t'$. Assume for contradiction that $t\in N(t')$. Then two cases can occur: either there is $1\leq i < k$ such that $t=s_1 s_2 \cdots s_{i-1} s_i s_{i-1}\cdots s_2 s_1$, or there is $1 \leq i < k$ such that $t= s_1 s_2 \cdots s_k s_{k-1} \cdots s_i s_{i+1} \cdots s_k s_{k-1} \cdots s_1$ (the case where $i=k$ yields $t=t'$).
	
	In the first case, we have $t'=tt't= s_1 s_2 \cdots s_{i-1} s_{i+1} \cdots s_k s_{k-1} \cdots s_{i+1} s_{i-1} \cdots s_1$, which is an expression for $t'$ in the elements of $S$ that is of strictly smaller length than $s_1 s_2\cdots s_k \cdots s_2 s_1$, a contradiction, since the latter was assumed to be $S$-reduced. 
	
	In the second case, we have $tt'= s_1 s_2 \cdots s_k s_{k-1} \cdots s_{i+1} s_{i-1} \cdots s_2 s_1$. But we also have $t't= s_1 \cdots s_{i-1} s_{i+1}\cdots s_k s_{k-1} \cdots s_2 s_1$. Since $tt'=t't$ we get $s_{i+1} \cdots s_k s_{k-1} \cdots s_i = s_i \cdots s_k s_{k-1}\cdots s_{i+1}$, yielding $s_{i} s_{i+1} \cdots s_k s_{k-1} \cdots s_{i+1} s_i=s_{i+1} \cdots s_k s_{k-1} \cdots s_{i+1}$, contradicting again the fact that $s_1 s_2 \cdots s_k s_{k-1} \cdots s_2 s_1$ is reduced.  
\end{proof}

\begin{rmq}
	Lemma~\ref{lem_ref_commute} can also be proven using root systems. Let $\Phi$ be the generalized root system attached to $(W,S)$. We have $\Phi=\Phi^+ \coprod (- \Phi^+)$, where $\Phi^+=\{\alpha_t~|~t\in T\}$ is the set of positive roots. In this setting, for $w\in W$ we have $$N(w)=\{t\in T~|~w(\alpha_t)\in(-\Phi^+)\}.$$ Let $t,t'$ satisfying the assumptions of Lemma~\ref{lem_ref_commute} and assume for contradiction that $t\in N(t')$. Then $t'(\alpha_t)\in(-\Phi^+)$. But $t'(\alpha_t)=\pm \alpha_{t'tt'}=\pm \alpha_{t}$, which forces $t'(\alpha_t)=-\alpha_t$. It follows that $\alpha_t$ is an eigenvector of $t'$ for the eigenvalue $-1$, hence it is proportional to $\alpha_{t'}$, yielding $\alpha_t=\alpha_{t'}$, a contradiction. 
\end{rmq}

\section{Proof of Theorem~\ref{thm_min}}

In this section, we prove Theorem~\ref{thm_min}. We keep the notation introduced in the previous section, recalling that we always assume that $\theta$ satisfies $\theta^2=\mathrm{id}$.

\medskip

We begin by proving Proposition~\ref{prop_esc}. 

\begin{proof}[{Proof of Proposition~\ref{prop_esc}}]
	
Let $i\in\{0, 1, \dots, k-1\}$. Set $y:=x_1 x_2 \cdots x_{i}$. We separate the proof into two cases depending on whether $x_{i+1}$ is a reflection or not. \medskip 

\textbf{$\bullet$~Case where $x_{i+1}\in \Theta_T$}. Since $x_{i+1}$ is a reflection, we want to show that $x_{i+1}\notin N(uy)$. Since $x_1 x_2\cdots x_{i+1}$ is $S_L^\theta$-reduced, we have that $x_1 x_2 \cdots x_i$ is also $S_L^\theta$-reduced, hence by Proposition~\ref{prop_muhlh} we have $$\ell(y x_{i+1})=\ell(x_1 x_2 \cdots x_{i+1})=\sum_{j=1}^{i+1} \ell(x_j)=\ell(x_1 \cdots x_i)+\ell(x_{i+1})=\ell(y)+\ell(x_{i+1}).$$ It follows that $x_{i+1} \notin N(y)$. Assume for contradiction that $x_{i+1}\in N(uy)$. Then since $N(uy)=N(y)\Delta (y^{-1} N(u) y)$, we have $x_{i+1}\in y^{-1} N(u) y$, hence $t:=y x_{i+1} y^{-1}\in N(u)$. Note that $t\in W_L^\theta$ since both $y$ and $x_{i+1}$ lie in $W_L^\theta$. But $t\in N(u)$ implies that $\ell(ut) < \ell(u)$, and since $ut\in u W_L^\theta$, this contradicts the fact that $u\in\mathcal{M}$.\medskip

\textbf{$\bullet$~Case where $x_{i+1}\in \Theta$}. Then $x_{i+1}$ is the longest element $w_{0,I}$ of a standard finite parabolic subgroup $W_I$, where $I=\{s,t\}$ is such that $W_I$ is of type $I_2(2m)$ for some $m\geq 1$. In particular $x_{i+1}$ has exactly two reduced expressions $(st)^m=(ts)^m$ in $W$. 

If $\ell(uy x_{i+1})>\ell(uy)$, then by Lemma~\ref{lem_long_gen} we have $uyx_{i+1} > uy$, which concludes the proof in that case. It therefore suffices to show that the case where $\ell(uy x_{i+1})< \ell(uy)$ leads to a contradiction. Hence assume that $\ell(uy x_{i+1})< \ell(uy)$. By Lemma~\ref{lem_long_gen} again, we have $uyx_{i+1} < uy$. Setting $x:=uy$, we consider the decomposition $x=x^I x_I$, where $x^I\in W^I$, $x_I\in W_I$, with respect to the standard dihedral parabolic subgroup $W_I$. Setting $v=x x_{i+1}$, since $x_{i+1}\in W_I$ we have $v^I=x^I$ and $v_I= x_I x_{i+1}$. We thus have $$\ell(x^I)+\ell(x_I)=\ell(x)>\ell(v)=\ell(v^I)+\ell(v_I)=\ell(x^I)+\ell(x_I x_{i+1}),$$ from what we deduce that $\ell(x_I) > \ell(x_I x_{i+1})$. Since $x_I$ and $x_I x_{i+1}$ have the same parity of length (because $x_{i+1}$ has even length), we must in fact have $\ell(x_I x_{i+1})\leq \ell(x_I)-2$. In particular, we have $x_I\neq 1$, and there is $r\in I=\{s,t\}$, say $r=s$ without loss of generality, such that $x_I s < x_I$. We thus have $\ell(x_Is)=\ell(x_I)-1$ since $s$ is a simple reflection and we deduce that  $$\ell(x_I x_{i+1}) < \ell(x_I s) < \ell(x_I).$$ Since $x_I x_{i+1}, x_I s$ and $x_I$ all lie in $W_I$ which is dihedral, we deduce that $$x_I x_{i+1} < x_I s < x_I.$$ This stays preserved when multiplying on the left by $x^I$, yielding $$v=x^I x_I x_{i+1} < x^I x_I s= x s < x^I x_I=x.$$ Note that $q:=s x_{i+1}\in T$ since $s x_{i+1}= s(ts)^{m}$. We thus have $v=xsq < xs < x$. Moreover, for length reasons, since $sq=qs$ we must have $xq < x$ (as $xq >x$ would contradict $xqs=xsq < x$ as $s$ is simple), hence both $s,q$ lie in $N(x)$. We now argue in a similar way as in the first case above to obtain a contradiction: since $x_1 x_2\cdots x_{i+1}=yx_{i+1}$ satisfies $\ell(y x_{i+1})=\sum_{j=1}^{i+1} \ell(x_j)$, then $y x_{i+1}$ has a reduced expression obtained by concatenating a reduced expression of $y$ and a reduced expression of $x_{i+1}=w_{0,I}$, hence $y\in W^I$. Since both $s$ and $q$ are reflections in $W_I$, we deduce that $s,q\notin N(y)$. As $x=uy$ and $s,q$ both lie in $N(x)$ but none of them lies in $N(y)$, using $N(x)=N(uy)=(y^{-1} N(u)y) \Delta N(y)$ we deduce that both $\widetilde{s}:=y s y^{-1}$ and $\widetilde{q}:=y q y^{-1}$ lie in $N(u)$. We thus have $u\widetilde{s} < u$, $u\widetilde{q} < u$. We have $\widetilde{q}\notin N(\widetilde{s})$ by Lemma~\ref{lem_ref_commute}, as $sq=qs$ implies that $\widetilde{s}\widetilde{q}=\widetilde{q}\widetilde{s}$. Since $N(u\widetilde{s})=(\widetilde{s} N(u)\widetilde{s}) \Delta N(\widetilde{s})$ and $\widetilde{q}\notin N(\widetilde{s})$, $\widetilde{q}\in N(u)$ (hence $\widetilde{q}=\widetilde{s}\widetilde{q}\widetilde{s}\in\widetilde{s} N(u) \widetilde{s}$), we get that $\widetilde{q}\in N(u\widetilde{s})$, hence $u \widetilde{s} \widetilde{q} < u \widetilde{s} < u$. But $\widetilde{s} \widetilde{q}=y sq y^{-1}=y x_{i+1} y^{-1}\in W_L^\theta$, hence $u \widetilde{s} \widetilde{q}\in u W_L^\theta$ with $\ell(u \widetilde{s}\widetilde{q}) <\ell(u)$, contradicting $u\in\mathrm{Min}(u)$.      
\end{proof}

We can now prove Theorem~\ref{thm_min}.

\begin{proof}[{Proof of Theorem~\ref{thm_min}}]
We apply Proposition~\ref{prop_esc} with $w=v$, $z=y$, and $x_i=y_i$ for all $i=1, \dots, k$, which is possible since $u\in\mathcal{M}$, $v\in u W_L^\theta$, and the expression $y_1 y_2 \cdots y_k$ is $S_L^\theta$-reduced. For all $i=0, \dots, k-1$, we thus get $\ell(u y_1 \cdots y_i)=\ell(u y_1 \cdots y_{i+1})$ or $u y_1 \cdots y_i<u y_1 \cdots y_{i+1}$, hence in all cases we have $\ell(u y_1 \cdots y_i)\leq \ell(u y_1 \cdots y_{i+1})$. We thus have $$\ell(u) \leq \ell(u y_1) \leq \dots \leq \ell(u y_1 \cdots y_i) \leq \dots \leq \ell(u y_1 \cdots y_{k-1}) \leq \ell(v).$$ But since $u, v\in\mathcal{M}\cap u W_L^\theta$, we have $\ell(u)=\ell(v)$, hence all inequalities in the above sequence are in fact equalities, which concludes the proof. \end{proof}

\begin{exple}
	
	Let $W = F_4$, $L=S=\{s_1,s_2,s_3,s_4\}$ and $\theta$ be the diagram automorphism of $L$ given by the following figure:
	
	
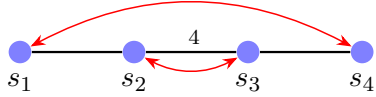
\begin{center}
		\begin{tikzpicture}[main_node/.style={circle,fill=blue!50,minimum size=0.5em,inner sep=3pt]}]
			
			\node[main_node] (1) at (0,0) {};
			\node[main_node] (2) at (1.5, 0)  {};
			\node[main_node] (3) at (3, 0) {};
			\node[main_node] (4) at (4.5, 0) {};
			
			\node at (0,-0.4) {$s_1$} ;
			\node at (1.5,-0.4) {$s_2$} ;
			\node at (3,-0.4) {$s_3$} ;
			\node at (4.5,-0.4) {$s_4$} ;
			
			\node at (2.25,0.2) {$\tiny{~_4}$} ;
			
			\draw[line width=0.3mm] (1) to (2);
			\draw[line width=0.3mm] (2) to (3);
			\draw[line width=0.3mm] (3) to (4);
			
			\draw [line width=0.2mm,Stealth-Stealth ,red] (1) to [bend left=28] (4);
			\draw [line width=0.2mm,Stealth-Stealth, red] (2) to [bend left=-28] (3);
		\end{tikzpicture}
	\end{center}
	
	We have $S_L^\theta=\{s_1 s_4, s_2 s_3 s_2 s_3\}$ and $W_L^\theta$ is a dihedral group of order $16$. There are $72$ classes in $W / W_L^{\theta}$,  each of them having $16$ elements.  Let $X \in W / W_L^{\theta}$ and let $u \in X$.  By computational experimentations with the software SageMath we found that $$|\text{Min}(u)| \in \{1,2,3,4,5,6,8,16\}.$$
	\begin{flushleft}
		
	\end{flushleft}
	
	We now give two examples of such classes where we see how the minimal elements are related by the elements of $S_L^{\theta} = \{s_1s_4, s_2s_3s_2s_3\}$ when multiplying on the right, as an illustration of Theorem~\ref{thm_min}.  Write $x = s_1s_4$ and $y = s_2s_3s_2s_3$. For simplicity we will denote a reduced expression $s_{i_1} s_{i_2} \cdots s_{i_k}$ of an element of $W$ simply by $i_1 i_2 \cdots i_k$. 
	\medskip
	
	1) If $u = 42312342$ then the minimal elements of $X$ are given in Figure \ref{First Figure min F4}.
	\begin{figure}[h!]
		\begin{tikzpicture}
			\node at (0,0) (1) {$42312342$};
			\node at (3,0) (2) {$42312321$};
			\node at (6,0)(3) {$43123121$} ;
			\node at (9,0)(4) {$43123412$} ;
			
			\draw[line width=0.3mm, Stealth-Stealth] (1) to (2);
			\draw[line width=0.3mm,Stealth-Stealth]  (2) to (3);
			\draw[line width=0.3mm,Stealth-Stealth]  (3) to (4);
			
			\node at (1.5,0.25)  {$\textcolor{red}{x}$};
			\node at (4.5,0.25)  {$\textcolor{red}{y}$};
			\node at (7.5,0.25) {$\textcolor{red}{x}$} ;
			
		\end{tikzpicture}
		\caption{Minimal elements of a class having 4 minimal elements.}
		\label{First Figure min F4}
	\end{figure}

	2) If $u = 343231234312$ then the minimal elements of $X$ are given in Figure \ref{Second Figure min F4}.  Note that in this particular class, every element has minimal length in its coset. The conclusion of Theorem~\ref{thm_min} is thus trivially verified in this case.   
	\begin{figure}[h!]
		\scalebox{.9}{\begin{tikzpicture}
			\node at (0,0) (1) {$343231234312$};
			\node at (4,0) (2) {$432343123121$};
			\node at (8,0) (3) {$432342312321$} ;
			\node at (12,0)(4) {$234323123432$} ;
			
			\draw[line width=0.3mm, Stealth-Stealth] (1) to (2);
			\draw[line width=0.3mm,Stealth-Stealth]  (2) to (3);
			\draw[line width=0.3mm,Stealth-Stealth]  (3) to (4);
			
			\node at (2,0.25)  {$\textcolor{red}{x}$};
			\node at (6,0.25)  {$\textcolor{red}{y}$};
			\node at (10,0.25) {$\textcolor{red}{x}$} ;

			\node at (-1,-1.5) (5) {$343231234123$};
			\node at (-1,-3) (6) {$342312341231$};
			\node at (-1,-4.5) (7) {$342312342312$} ;
			\node at (-1,-6)(8) {$312343123121$} ;
			
			\draw[line width=0.3mm, Stealth-Stealth] (1) to (5);
			\draw[line width=0.3mm,Stealth-Stealth]  (5) to (6);
			\draw[line width=0.3mm,Stealth-Stealth]  (6) to (7);
			\draw[line width=0.3mm,Stealth-Stealth]  (7) to (8);
			
			\node at (-0.7,-0.7)  {$\textcolor{red}{y}$};
			\node at (-1.3,-2.25)  {$\textcolor{red}{x}$};
			\node at (-1.3,-3.75)  {$\textcolor{red}{y}$};
			\node at (-1.3,-5.25) {$\textcolor{red}{x}$} ;

			\node at (0,-7.5) (9) {$312342312321$};
			\node at (4,-7.5) (10) {$123423123432$};
			\node at (8,-7.5) (11) {$123423123423$} ;
			\node at (12,-7.5)(12) {$231234231231$} ;
			
			\draw[line width=0.3mm, Stealth-Stealth] (8) to (9);
			\draw[line width=0.3mm,Stealth-Stealth]  (9) to (10);
			\draw[line width=0.3mm,Stealth-Stealth]  (10) to (11);
			\draw[line width=0.3mm,Stealth-Stealth]  (11) to (12);
			
			\node at (-0.7,-6.8)  {$\textcolor{red}{y}$};
			\node at (2,-7.75)  {$\textcolor{red}{x}$};
			\node at (6,-7.75) {$\textcolor{red}{y}$} ;
			\node at (10,-7.75) {$\textcolor{red}{x}$} ;

			\node at (13,-1.5) (13) {$234323123423$};
			\node at (13,-3) (14) {$423123431231$};
			\node at (13,-4.5) (15) {$423123432312$} ;
			\node at (13,-6) (16) {$231234323121$} ;
			
			\draw[line width=0.3mm, Stealth-Stealth] (4) to (13);
			\draw[line width=0.3mm,Stealth-Stealth]  (13) to (14);
			\draw[line width=0.3mm,Stealth-Stealth]  (14) to (15);
			\draw[line width=0.3mm,Stealth-Stealth]  (15) to (16);
			\draw[line width=0.3mm,Stealth-Stealth]  (16) to (12);
			
			\node at (12.7,-0.7)  {$\textcolor{red}{y}$};
			\node at (13.3,-2.25)  {$\textcolor{red}{x}$};
			\node at (13.3,-3.75)  {$\textcolor{red}{y}$};
			\node at (13.3,-5.25) {$\textcolor{red}{x}$} ;
			\node at (12.7,-6.8) {$\textcolor{red}{y}$} ;

		\end{tikzpicture}}
		\caption{Minimal elements of a class having 16 minimal elements.}
		\label{Second Figure min F4}
	\end{figure}
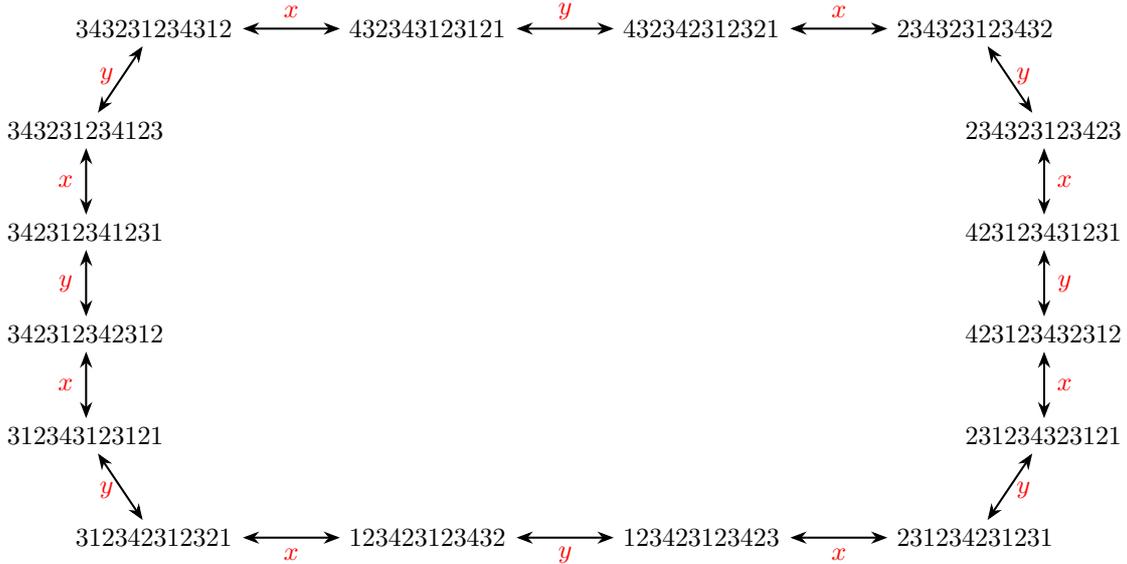
	
\end{exple}

\section{Proof of Theorem~\ref{main_mm}}

\begin{proof}[Proof of Theorem~\ref{main_mm}]
Let $u\in\mathcal{M}$. To show the result, it suffices to show the following: for all $k\geq 0$ and all $x_1, x_2, \dots, x_k\in S_L^\theta$ such that $x_1 x_2 \cdots x_k$ is $S_L^\theta$-reduced, there is $w\in \mathcal{M} \cap u W_L^\theta$ such that $w \leq u x_1 x_2 \cdots x_k$. Indeed, for every $x\in W$, it then suffices to choose $u\in \mathcal{M} \cap x W_L^\theta$, to choose an $S_L^\theta$-reduced expression $x_1 x_2\cdots x_k$ of $y:=u^{-1}x$ and apply the above result to $x=u x_1 x_2 \cdots x_k$ to find $w\in\mathcal{M} \cap x W_L^\theta$ such that $w \leq x$. 

The advantage of the above reformulation is that it allows one to argue by induction on $k$. For $k=0$ the result is trivial since one can take $w=u$. 

Hence assume that $k\geq 1$. By induction there is $w'\in \mathcal{M} \cap u W_L^\theta$ such that $w' \leq u x_1 x_2 \cdots x_{k-1}$. By Proposition~\ref{prop_esc}, the case where $\ell(ux_1 x_2 \cdots x_k)<\ell(u x_1 x_2 \cdots x_{k-1})$ cannot appear: we have either $u x_1 x_2 \cdots x_{k-1} < u x_1 x_2 \cdots x_k$, or $\ell(u x_1 x_2 \cdots x_{k-1})=\ell(u x_1 x_2 \cdots x_k)$. In the first case we are done with $w:=w'$, since $$w' \leq u x_1 x_2 \cdots x_{k-1}< u x_1 x_2 \cdots x_k.$$

Hence assume that $\ell(u x_1 x_2 \cdots x_{k-1})=\ell(u x_1 x_2 \cdots x_k)$. We have $w' \leq u x_1 x_2 \cdots x_{k-1}$ and $\ell(u x_1 x_2 \cdots x_{k-1})=\ell(u x_1 x_2 \cdots x_{k-1} x_k)$ with $x_k\in S_L^\theta$, hence by Lemma~\ref{lem_corr}, there is $w\in\{w', w' x_k\}$ such that $w \leq u x_1 x_2 \cdots x_{k-1} x_k$ and $\ell(w) \leq \ell(w')$. Since $w$ and $w'$ lie in the same coset modulo $W_L^\theta$ and $w'\in \mathcal{M}$, we must have $w\in \mathcal{M} \cap u W_L^\theta$, which concludes the proof. 
\end{proof}

\end{document}